\documentclass[12pt,a4paper]{amsart}

\usepackage{amsmath,amsthm,amssymb,latexsym,a4wide,epic,eepic,bbm}

\newtheorem{theorem}{Theorem}
\newtheorem{lemma}[theorem]{Lemma}
\newtheorem{corollary}[theorem]{Corollary}
\newtheorem{proposition}[theorem]{Proposition}

\newtheorem{example}[theorem]{Example}
\newtheorem{remark}[theorem]{Remark}
\usepackage[all]{xy}

\usepackage[active]{srcltx}

\usepackage{enumerate}
\numberwithin{equation}{section}

\begin{document}
\title{Partial monoid actions and a class of restriction semigroups}
\author{Ganna Kudryavtseva}
\thanks{The author was partially supported by ARRS grant P1-0288}


\begin{abstract}
We study classes of proper restriction semigroups determined by properties of partial actions underlying them. These properties include strongness, antistrongness, being defined by a homomorphism, being an action etc.   Of particular interest is the class determined by  homomorphisms, primarily because we observe that its elements, while being close to semidirect products,  serve as mediators  between general restriction semigroups and semidirect products or $W$-products in an embedding-covering construction. It is remarkable that this class does not have an adequate analogue if specialized to inverse semigroups.  $F$-restriction monoids of this class, called  ultra $F$-restriction monoids, are determined by homomorphisms from a monoid $T$ to the Munn monoid of a semilattice $Y$. We show that these are precisely the monoids $Y*_mT$ considered by Fountain, Gomes and Gould.  We obtain a McAlister-type presentation for the class given by strong dual prehomomorphisms and apply it to construct an embedding of ultra $F$-restriction monoids, for which the base monoid $T$ is free, into  $W$-products of semilattices by monoids. Our approach yields new and simpler proofs of two recent embedding-covering results by Szendrei. \end{abstract}
\maketitle

\section{Introduction}\label{s1}
Restriction semigroups, also known as weakly $E$-ample semigroups, are non-regular generalizations of inverse semigroups. These are semigroups with two additional unary operations which mimic the operations $a\mapsto a^{-1}a$ and $a\mapsto aa^{-1}$ on an inverse semigroup. Various aspects of restriction semigroups and their one-sided analogues have been extensively studied in the literature, see, e.g., \cite{FGG,G,H2} and references therein.

Proper restriction semigroups are analogues of $E$-unitary inverse semigroups which play a central role in the theory of inverse semigroups and its applications. Generalizing corresponding results for inverse and ample semigroups \cite{PR,KL,Law2}, Cornock and Gould \cite{CG} gave a structure theorem for proper restriction semigroups in terms of double partial actions of monoids on semilattices. This can be readily reformulated in terms of only one partial action, since each of the two partial actions is determined by the other one.  In the present paper we consider classes of proper restriction semigroups determined by the properties of this partial action. We show that the partial action is strong or antistrong if and only if the restriction semigroup satisfies a technical condition arising in \cite{CG}. $W$-products of  semilattices by monoids correspond to the situation where the partial action is an action, and semidirect products form their subclass corresponding to actions by automorphisms. More importantly, we single out a class of proper restriction semigroups determined by homomorphisms. We call elements of this class ultra proper restriction semigroups. This is a rich and important class, since its elements, while being close to semidirect products, arise as mediators between general restriction semigroups and $W$-products or semidirect products in an embedding-covering result. It is interesting that this class does not have an adequate analogue if specialised to the inverse case, ultra proper inverse semigroups being precisely semidirect products of semilattices by groups. This discrepancy between restriction and inverse semigroups is well illustrated by the fact that free  restriction monoids and semigroups are ultra proper, whereas free inverse monoids and semigroups are not. The subclass of ultra proper restriction semigroups which are also $F$-restriction (where $F$-restriction has a similar meaning as $F$-inverse in the inverse semigroup theory)  is shown to be equal to the class of monoids  $Y*_m T$ introduced by Fountain, Gomes and Gould~in~\cite{FGG}. Alternatively, this subclass can be described by the property that the underlying homomorphism has its range in the Munn monoid of the semilattice. We call elements of this subclass ultra $F$-restriction monoids.

Based on ideas from \cite{M} and \cite{MS}, we construct a globalization of a strong partial action underlying a proper restriction semigroup and obtain a McAlister-type theorem. We then specialize this construction to partial actions defining ultra $F$-restriction monoids $M(T,Y)$, where the base monoid $T$ is free. As a result, we obtain a semilattice $X$ and an action of $T$ on this semilattice such that the $W$-product $W(T,X)$ can be formed and, moreover, the initial monoid $M(T,Y)$ embeds into $W(T,X)$.  This construction is inspired by Szendrei's embedding of the free restriction monoid into a $W$-product \cite{S1} and generalizes it. We apply the constructed embedding in the following setting.  Let $S$ be an $A$-generated restriction monoid and $T=A^*$ be the $A$-generated free monoid. Modifying a construction from \cite{FGG}, we produce a partially defined action of $T$ on the semilattice of projections $Y$ of $S$ defining an ultra $F$-restriction (and even ample) monoid $M(T, Y)$, which  covers  $S$. 
If $S$ is a restriction semigroup, a variation of this construction produces an ultra proper restriction semigroup $M(T, Y)$, which covers $S$ and is a restriction subsemigroup of the ultra $F$-restriction semigroup $M(T, Y^1)$. 
The globalization construction enables us to embed  $M(T, Y)$ into a $W$-product $W(T, X)$, yielding new and simpler proofs of two embedding-covering results by Szendrei \cite{S1,S}. The first result \cite{S1} states that any restriction semigroup $S$ has a proper ample cover embeddable into a $W$-product. Our argument proving this is that $M(T, Y)$ provides such a cover, and we emphasize that the cover is ultra proper.  The second result \cite{S} states that any restriction semigroup can be embedded into an almost left factorizable restriction semigroup, that is, a $(2,1,1)$-quotient of a $W$-product. To show this, we construct a projection separating $(2,1,1)$-congruence $\kappa$ on $W(T, X)$, which extends the congruence on $M(T, Y)$ mapping it onto $S$, so that $S$ embeds into $W(T,X)/\kappa$. 

We conclude the introduction by pointing out that, after an early version of this paper existed, the author learned that ultra proper restriction semigroups and ultra $F$-restriction monoids were independently introduced and studied (from a somewhat different perspective) by Peter Jones in \cite{J}
 under the names almost perfect restriction semigroups and perfect restriction monoids, respectively. This terminology is motivated by an elegant characterization noticed in \cite{J} that a proper restriction semigroup is ultra proper if and only if the least congruence identifying all projections is perfect meaning that the product of classes is again a whole class. 

\section{Preliminaries}
\subsection{Restriction semigroups} In this section we recall the definition and basic properties of restriction semigroups \cite{G, GouldS, S1,S}. Further details, including a different approach to restriction semigroups, via generalized Green's relations, can be found in \cite{G,H2}.

A {\em restriction semigroup} is an algebra $(S, \cdot, ^*, ^+)$, where $(S,\cdot)$ is a semigroup and $^*$ and $^+$ are unary operations satisfying the following identities (here and in the sequel the multiplication in $S$ is denoted just by juxtaposition):
\begin{equation}\label{eq:axioms:star}
xx^*=x, \,\,\, x^*y^*=y^*x^*, \,\,\, (xy^*)^*=x^*y^*, \,\,\, x^*y=y(xy)^*;
\end{equation}
\begin{equation}\label{eq:axioms:plus}
x^+x=x, \,\,\, x^+y^+=y^+x^+, (x^+y)^+=x^+y^+, \,\,\, xy^+=(xy)^+x;
\end{equation}
\begin{equation}\label{eq:axioms:common}
(x^+)^*=x,\,\,\, (x^*)^+=x.
\end{equation}
If a restriction semigroup has an identity element $1$ with respect to multiplication, it follows from $1^*1=1$ and the dual axiom involving $^+$ that 
$1^*=1^+=1$. A restriction semigroup possessing an identity is called a {\em restriction monoid}. Restriction semigroups form a variety of algebras of type $(2,1,1)$ and restriction monoids form a variety of algebras of type $(2,1,1,0)$. 

The following identities follow from the axioms and will be frequently  used in the sequel (usually without reference):
\begin{equation}\label{eq:consequences}
(xy)^*=(x^*y)^*,\,\,\, (xy)^+=(xy^+)^+.
\end{equation}

A   homomorphism of restriction semigroups is required to preserve the multiplication and the operations  $^*$ and $^+$, that is to be a homomorphism of $(2,1,1)$-algebras, and not only a semigroups homomorphism as the term might seem to suggest. For emphasis, we will often use the terms $(2,1,1)$-homomorphisms, $(2,1,1)$-congruence, $(2,1,1)$-subalgebra, etc. Similar remarks apply when restriction semigroups are replaced by restriction monoids.

Let $S$ be a restriction semigroup. It follows from \eqref{eq:axioms:common} that 
$$
\{x^*\colon x\in S\}=\{x^+\colon x\in S\}.
$$
We denote this set by $E$. It can be easily deduced that $E$ is closed with respect to the multiplication, is a semilattice and also $x^*=x^+=x$ for all $x\in E$. It follows that $E$ is a $(2,1,1)$-subalgebra of $S$. It is called the {\em semilattice of projections} of $S$ and is denoted by $P(S)$. Note that a projection is necessarily an idempotent, but a restriction semigroup may contain idempotents which are not projections.

As it was mentioned in the introduction, restriction semigroups generalize inverse semigroups (considered as $(2,1,1)$ algebras where $x^*=x^{-1}x$ and $x^+=xx^{-1}$). A crucial result about inverse semigroups is the Ehresmann-Nambooripad-Schein theorem \cite{Law} which says that the category of inverse  semigroups is isomorphic to the category of inductive groupoids. It is of fundamental importance that this theorem can be extended to restriction semigroups: they form a category which is  isomorphic to the category of inductive categories \cite{H2}. This fact brings up a crucial insight as well as provides an evidence that restriction semigroups are naturally arising algebraic objects.

A restriction semigroup $S$ is called {\em ample} if for all $a,b,c\in S$:
$$
ac=bc \Rightarrow ac^+=bc^+ \,\,\,\, \text{ and } \,\,\,\, ca=cb\Rightarrow c^*a=c^*b.
$$
Under the correspondence between restriction semigroups and inductive categories, ample semigroups correspond to cancellative inductive categories.

Let $S$ be a restriction semigroup and  $E=P(S)$. For $a,b\in S$ we set $a\leq b$ provided that there is $e\in E$
with $a=eb$. This relation is a partial order called the {\em natural partial order} on $S$.  The following properties of restriction semigroups related to the partial order will be used throughout the paper.

\begin{lemma}\label{lem:lem1} Let $S$ be a restriction semigroup, $E=P(S)$ and $a,b\in S$. Then
\begin{enumerate}
\item \label{i:b1}
  $a\leq b$ if and only if  $a=bf$ for some  $f\in E$. 
\item \label{i:b2} $a\leq b$ if and only if $a=ba^*$ if and only if $a=a^+b$. 
 \item \label{i:b3}If  $ea=a$  ($ae=a$), where $e\in E$, then $e\geq a^+$ (respectively, $e\geq a^*$). 
\item \label{i:b4} $a\geq ae, ea$ for any $e\in E$.
\item \label{i:b5} The order $\geq$ is compatible with the multiplication, that is, $a\geq b$ implies $ac\geq bc$ and $ca\geq cb$ for any $c\in S$. 
\item\label{i:b6}  The order $\geq$ is compatible with the unary operations, that is, $a\geq b$ implies $a^*\geq b^*$ and $a^+\geq b^+$. 
  \end{enumerate}
\end{lemma}

Let $\sigma$ denote the least congruence on a restriction semigroup $S$, which identifies all elements of $P(S)$. It is well known that the following statements are equivalent: (i) $a\mathrel{\sigma} b$;  (ii) there is $e\in E$ such that $ea=eb$; (iii) there is $e\in E$ such that $ae=be$.

A restriction semigroup $S$ is called {\em proper} if the following two conditions hold:
\begin{align*}
& \text{for any } a,b\in S: \text{ if } a^*=b^*  \text{ and } a\mathrel{\sigma} b \text{ then } a=b,\\
& \text{for any } a,b\in S: \text{ if } a^+=b^+  \text{ and }  a\mathrel{\sigma} b \text{ then } a=b.
\end{align*}

A $(2,1,1)$-morphism $\varphi: T\to S$ of restriction semigroups is called {\em projection separating} if $\varphi(e)\neq \varphi(f)$ for any $e,f\in P(T)$ such that $e\neq f$.  A restriction semigroup $T$ is called a {\em cover} of a restriction semigroup $S$ if there is a surjective projection separating $(2,1,1)$-morphism $\varphi: T\to S$.

\subsection{Actions} Let $T$ be a monoid with identity $1$ and $X$ be a set. A {\em left action} of $T$ on $X$ is a map $T\times X\to X$, $(t,x)\mapsto t*x$, such that 
 $1*x=x$ for all $x\in X$ and 
 $s*(t*x) = (st)*x$ for all $s,t\in T$ and $x\in X$.
A {\em right action} of $T$ on $X$  is defined dually. 
A left action of $T$ on $X$ can be equivalently given by a monoid homomorphism $\varphi: t\mapsto \varphi_t$, $\varphi_t(x)=t*x$,  from $T$ to the full transformation monoid ${\mathcal T}(X)$. A right action can be given by an anti-homomorphism in a similar way.
 
A map $\pi:X\to Y$ between posets is called {\em order-preserving} if $x\leq y$ implies $\pi(x)\leq \pi(y)$ for all $x,y\in X$. It is called an {\em order-embedding} if $x\leq y$ holds if and only if $\pi(x)\leq \pi(y)$ for all $x,y\in X$.
An order-embedding is necessarily an injective map. A bijective order-embedding is called an {\em order-isomorphism}.
If $X=Y$,  order-isomorphisms are called {\em order-automorphisms.}
A left action of $T$ on a poset $X$, given by a homomorphism $\varphi:T\to {\mathcal T}(X)$, is called {\em order-preserving} (an action {\em  by order-embeddings} or {\em  by order-automorphisms}) if for each $t\in T$ the transformation $\varphi_t$ is order-preserving (resp. an order-embedding or an order-automorphism). 

Let  $\cdot$ and $*$ be left actions of monoids $T$ and $T'$ on posets $X$ and $X'$, respectively. These actions are called {\em isomorphic} if there are a monoid isomorphism $\alpha:T\to T'$ and an order-isomorphism $\tau:X\to X'$ such that $\alpha(t)*(\tau(x))=\tau(t\cdot x)$ for all $t\in T$ and $x\in X$. 

These concepts can be readily adapted to right actions.

\subsection{$W$-products and semidirect products} We now recall the (left hand version of the) construction of a $W$-product of a semilattice by a monoid, which is a generalization of the construction of a semidirect product of a semilattice by a group and shares several of its important properties \cite{GomesS, S1,S}.

Let $T$ be a monoid and $*$ be a left action of $T$ on a semilattice $Y$ by order embeddings. 
Assume that the ranges of actions of the elements of $T$ are order ideals of $Y$, that is, if $x\leq t*y$ for $t\in T$ and $x,y\in Y$, then there is $z\in Y$ such that $x=t*z$.
We set
$$
W(T,Y)=\{(t*y,t)\in Y\times T\colon y\in Y, t\in T\}
$$
and define the multiplication and the unary operations $^*$ and $^+$ on $W(T,Y)$ by
\begin{align*}
& (t*y,t)(s*x,s)=(t*y\wedge (ts)*x,ts),\\
& (t*y,t)^*=(y,1); \,\, (t*y,t)^+=(t*y,1).
\end{align*}

Then $W(T,Y)$ is a proper restriction semigroup and its semilattice of projections is 
$$\overline{Y}=\{(y,1)\colon y\in Y\},
$$
which is isomorphic to $Y$. Furthermore, in $W(T,Y)$ we have $(t*y,t)\mathrel{\sigma} (s*x,s)$ if and only if $s=t$, and $W(T,Y)/\sigma \simeq T$ via the map $\sigma(t*y,t)\mapsto t$, where $\sigma(t*y,t)$ is the $\sigma$-class of $(t*y,t)$.
The semigroup $W(T,Y)$ is a monoid if and only if $Y$ has an identity.

Remark that in the literature (e.g. in \cite{GomesS, S1,S}) a $W$-product $W(T,Y)$ is usually defined starting from an action of $T$ on $Y$ by injective endomorphisms such that the range of action of each element of $T$ is an order ideal of $Y$,
where an {\em endomorphism} of $Y$ is a transformation preserving the operation $\wedge$. An endomorphism of a semilattice is obviously order-preserving. Conversely, if $\alpha$ is an order-embedding of a semilattice $Y$ and its range is an order ideal in $Y$ then $\alpha$ is easily seen to respect $\wedge$, i.e., to be an endomorphism.
Thus our definition of a $W$-product is equivalent to the usual one.

The semigroup $W(T,Y)$ is called a {\em semidirect product} of $Y$ by $T$ and is denoted by $T\ltimes Y$ provided that the action $*$ is by order automorphisms. Clearly, in this case $W(T,Y)$, as a set, equals just $Y\times T$. 

\subsection{Free restriction monoids and semigroups}\label{sub:free} Free restriction monoids and semigroups will appear throughout the paper  as  examples illustrating our constructions. We briefly recall their structure, and refer the reader to \cite{FGG,S,S1} for further details.  Let $A$ be a set, $A^*$ the free monoid over $A$ and let $1$ be the empty word. By $F\mathcal{G}(A)$ we denote the free group over $A$. The elements of $F\mathcal{G}(A)$ are reduced words over $A\cup A^{-1}$. If $u,v\in F\mathcal{G}(A)$, their product $\mathrm{red}(uv)$ is the reduced word equivalent to the word $uv$ obtained by the concatenation of $u$ and $v$. The set $F\mathcal{G}(A)$ is partially ordered by the {\em prefix order} $\leq_p$. Let ${\mathcal Y}'$ be the set of all finite order ideals of $(F\mathcal{G}(A),\leq_p)$, and let  ${\mathcal Y}={\mathcal Y}'\setminus \{1\}$.
The group $F\mathcal{G}(A)$ acts  on the left of its powerset by 
$$
v*S = \{{\mathrm{red}}(vs)\colon s\in S\}.
$$
Let 
$$
\begin{array}{ll}
{\mathcal{X}}' =F\mathcal{G}(A)*{\mathcal{Y}}', \quad & {\mathcal{X}} =F\mathcal{G}(A)*{\mathcal{Y}},\\
{\mathcal{Q}}'  =A^**{\mathcal{Y}}',  \quad & {\mathcal{Q}} =A^**{\mathcal{Y}}.
\end{array}
$$
With respect to the reverse inclusion order  ${\mathcal{X}}'$, ${\mathcal{X}}$, ${\mathcal{Q}}'$ and  ${\mathcal{Q}}$ are semilattices.  Clearly,  ${\mathcal{X}}'$ and ${\mathcal{X}}$ are invariant under the action of $F\mathcal{G}(A)$ and ${\mathcal Q}'$, ${\mathcal Q}$ are invariant under the action of  $A^*$. 
There is also a right action $\bullet$ of $F\mathcal{G}(A)$ on its powerset given by 
$S\bullet v=v^{-1}*S$ and both ${\mathcal{X}}'$ and ${\mathcal{X}}$ are invariant under this action. However,  neither ${\mathcal Q}'$ nor ${\mathcal Q}$ is invariant under $\bullet$. The left action $*$ of $A^*$ on ${\mathcal{Q}}'$  satisfies the requirements for forming of the $W$-product $W(A^*, {\mathcal{Q}}')$. 
The set 
$$
F_W{\mathcal{RM}}(A)=\{(t*y,t)\in W(A^*, {\mathcal{Q}}')\colon y\in {\mathcal{Y}}' \text{ and } t*y \in {\mathcal{Y}}'\}
$$
forms  a $(2,1,1)$-subalgebra of $W(A^*,{\mathcal{Q}}')$ which is isomorphic to the free restriction monoid $F{\mathcal{RM}}(A)$ over $A$. We refer to it as the {\em Szendrei's model} of the free restriction monoid. Note that 
$$P(F_W{\mathcal{RM}}(A))=\{(y,1) \colon y\in {\mathcal{Y}}'\}
$$
is a semilattice isomorphic to ${\mathcal{Y}}'$ via the map $(y,1)\mapsto y$; $(t*y,t)\mathrel{\sigma} (s*x,s)$ if and only if $t=s$ and $W(A^*,{\mathcal{Q}}')/\sigma \simeq A^*$ via the map $\sigma(t*y,t)\mapsto t$.

Similarly, we can form the $W$-product $W(A^*, {\mathcal{Q}})$ and construct the Szendrei's model of the free restriction semigroup $F{\mathcal{RS}}(A)$ over $A$ as its $(2,1,1)$-subalgebra  $F_W{\mathcal{RS}}(A)$ by replacing ${\mathcal{Q}}'$ by ${\mathcal{Q}}$ and ${\mathcal{Y}}'$ by ${\mathcal{Y}}$ in the definition of $F_W{\mathcal{RM}}(A)$.


\section{Classes of proper restriction semigroups}\label{s:classes}
\subsection{Partial actions and the Cornock-Gould structure theorem}\label{s:cg}
Let $T$ be a monoid with the identity element $1$ and let $X$ be a set.
A {\em left partial action} of $T$ on $X$ is a partial map $T\times X\to X$, $(t,x)\mapsto t\cdot x$, such that 
\begin{enumerate}[(LP1)]
\item For all $x\in X$: $1\cdot x$ is defined and $1\cdot x=x$.
\vspace{0.1cm}
\item For all $s,t\in T$ and $x\in X$: if $t\cdot x$ is defined and $s\cdot (t\cdot x)$ is defined then $(st)\cdot x$ is defined and  $s\cdot (t\cdot x) = (st) \cdot x$. 
\end{enumerate}
A {\em right partial action} of $T$ on $X$ is defined dually.   

A left partial action of a monoid $T$ on a set $X$ can be looked at as a  dual prehomomorphism from $T$ to the partial transformation monoid ${\mathcal{PT}}(X)$.  Recall that a map $\varphi\colon  T\to {\mathcal{PT}}(X)$, $t\mapsto \varphi_t$, is called  a {\em dual prehomomorphism} if $\varphi_1$ is the identity transformation and $\varphi_{st}$ is an extension of $\varphi_s\varphi_t$, that is,  $x\in{\mathrm{dom}}(\varphi_s\varphi_t)$ implies $x\in{\mathrm{dom}}(\varphi_{st})$ and $\varphi_s\varphi_t(x)=\varphi_{st}(x)$. {\em Order-preserving} partial actions, partial actions by {\em order-isomorphisms} and {\em isomorphic partial actions} on posets are defined similarly to the analogous notions for actions.

Let now $Y$ be a semilattice and $T$ a monoid acting on $Y$ partially on the left via $\cdot$, and let $\varphi:T\to \mathcal{PT}(Y)$ be the corresponding dual prehomomorphism. Assume that for every $t\in T$ the map $\varphi_t$  satisfies the following axioms: 
\begin{enumerate}
\item[(A)] ${\mathrm{dom}}(\varphi_t)$ and  ${\mathrm{ran}}(\varphi_t)$ are order ideals of $Y$.
\vspace{0.1cm}
\item[(B)] $\varphi_t:{\mathrm{dom}}(\varphi_t)\to {\mathrm{ran}}(\varphi_t)$ is an order-isomorphism.
\vspace{0.1cm}
\item[(C)] ${\mathrm{dom}}(\varphi_t)\neq \varnothing$.
\end{enumerate}


\begin{lemma} The assignment $t\to \varphi_t^{-1}$ defines a right partial  action, $\circ$, of $T$ on $Y$ and its defining dual antiprehomomorphism $\psi$, given by $\psi_t=\varphi_t^{-1}$, satisfies  axioms (A), (B) and~(C). 
\end{lemma}

\begin{proof} It is immediate that  $y\circ 1=y$ for all $y\in Y$. Assume that $y\circ s$ and $(y\circ s)\circ t$ are defined. Let $x=y\circ s$ and $z=x\circ t$. Then $y=s\cdot x$ and $x=t\cdot z$, whence $y=s\cdot (t\cdot z)= st\cdot z$ since $\cdot$ is a partial action. Thus $y\circ st$ is defined and $y\circ st=z$. Axioms (A), (B), (C) for $\psi$ are straightforward to verify.
\end{proof}

We say that the partial actions $\cdot$ and $\circ$ are {\em reverse} to each other. It is immediate that for every $t\in T$ and $y\in Y$: 
$$
\text{ if } t\cdot y \text{ is defined then }  (t\cdot y)\circ t \text{ is defined and } (t\cdot y)\circ t=y;
$$
$$
\text{ if }  y\circ t \text{ is defined then }  t\cdot (y \circ t) \text{ is defined and } t \cdot (y \circ t)=y.
$$

We remark that in \cite[Section 4]{CG} a pair of partial actions (namely, a left partial action and a right partial action) of a monoid $T$ on a semilattice $Y$ satisfying certain conditions is considered. Notice that each of these two partial actions determines the other one: the first two conditions, (A) and (B), from \cite[Section 4]{CG} say precisely that for each $t\in T$ the two maps on $Y$  induced by $t$ are mutually inverse (and consequently, each of these maps is injective). Therefore, the setting from \cite[Section 4]{CG} is equivalent to ours. It follows in particular that the partial actions $\cdot$ and $\circ$ respect the operation $\wedge$ on $Y$, as it is established in \cite[Proposition 4.1]{CG}.

Given a left partial action $\cdot$ of $T$ on $Y$ such that axioms (A), (B) and (C) hold, we set
$$
M(T,Y)=\{(y,t)\in Y\times T\colon  y\circ t \text{ is defined}\}
$$ 
and define the multiplication on $M(T,Y)$  by
$$
(x,s)(y,t)=(s\cdot ((x\circ s)\wedge y), st).
$$

With respect to this multiplication, $M(T,Y)$ is a semigroup. Furthermore, elements of the form $(y,1)$, where $y$ runs through $Y$, form a subsemilattice $\overline{Y}$ isomorphic to $Y$. It is shown in \cite{CG} that if one defines the unary operations $^*$ and $^+$  by
$$(y,t)^*=(y\circ t,1), \,\, (y,t)^+=(y,1),$$
$M(T,Y)$ becomes a proper restriction semigroup with $\overline{Y}=P(M(T,Y))$. 

Conversely, let $S$ be a proper restriction semigroup and consider a left partial action $\cdot$ of $T=S/\sigma$ on $E=P(S)$ given for  $t\in S/\sigma$ and $e\in E$ by
\begin{equation}\label{eq:und}
t\cdot e \text{ is defined if and only if there exists } a\in t \text{ such that } a^*\geq e, 
\end{equation}
 in which case  $t\cdot e = (ae)^+$. Observe that $a^*\geq e$ implies $(ae)^*=e$, so that  $a'=ae\in t$ with $(a')^*=e$. We will use this fact in the sequel without further mention.
The partial action $\cdot$ satisfies axioms (A), (B), (C) and the following structure theorem holds:

\begin{theorem}[Cornock and Gould \cite{CG}] \label{th:CG} Any  proper restriction semigroup $S$  is isomorphic to $M(S/\sigma, P(S))$.
\end{theorem}

We refer to the left partial action $\cdot$  given by \eqref{eq:und} and its reverse right partial action $\circ$ as the partial actions  {\em underlying}~$S$.

\subsection{Partial actions and $W$-products}\label{sub:w} Let $W(T,Y)$ be a $W$-product and  $*$ the left action of $T$ on $Y$ defining it. It is immediate that axioms (A), (B) and (C) are satisfied by $*$ so that the semigroup $M(T,Y)$ may be formed. By construction, the sets $M(T,Y)$ and $W(T,Y)$ coincide and so do the unary operations $^*$ and $^+$ on these sets. A direct verification (or an application of  Lemma~\ref{lem:products}) shows that the products on $M(T,Y)$ and $W(T,Y)$ coincide, too. It follows that $M(T,Y)=W(T,Y)$ as $(2,1,1)$-algebras. If $W(T,Y)=T\ltimes Y$, that is, if 
$*$ is an action by automorphisms, the underlying left and right partial actions of $T\ltimes Y$  are actions.

Conversely, let $*$ be a left action of $T$ on $Y$ which satisfies axioms (A), (B) and (C). Then $*$ satisfies the requirements needed to form the $W$-product $W(T,Y)$. If $*$ acts by automorphisms we have $W(T,Y)=T\ltimes Y$. We obtain the following statement.

\begin{proposition}\label{prop:w_prod} 
$W$-products of semilattices by monoids (defined by left actions) are precisely the proper restriction semigroups whose underlying left partial action is an action. In particular, semidirect products of semilattices by monoids are precisely the proper restriction semigroups whose underlying left partial action is an action by automorphisms (or, equivalently, both of whose underlying left and right partial actions are actions).
\end{proposition}

\subsection{$F$-restriction monoids}
By analogy with inverse semigroups, we call a restriction semigroup an $F$-{\em restriction semigroup} if every $\sigma$-class has a maximum element.

\begin{lemma}\label{lem:f_rest} Let $S$ be an $F$-restriction semigroup. Then $S$ is proper.  Consequently, $S$ is necessarily a monoid with the identity being the maximum projection.
\end{lemma}

\begin{proof} For $a\in S$ let $m(a)$ be the maximum element in the $\sigma$-class of $a$. Assume that $a\mathrel{\sigma} b$ and $a^+=b^+$. Then $m(a)=m(b)$ and since $a\leq m(a)$, $b\leq m(b)$ we obtain that
$a=a^+m(a)=b^+m(b)=b$. Similarly, one can see that $a\mathrel{\sigma} b$ and $a^*=b^*$ imply $a=b$.
\end{proof}

Recall that the {\em Munn semigroup} $T_Y$ of a semilattice $Y$ is the semigroup of all order-isomorphism between principal order ideals of $Y$ under composition. This is an inverse semigroup contained in ${\mathcal I}(Y)$.

\begin{lemma}\label{lem:f_rest1} Let $T$ be a monoid acting partially on the left of a semilattice $Y$  so that $M(T,Y)$ can be formed and let $\varphi:T\to {\mathcal{I}}(Y)$ be the corresponding dual prehomomorphism. The following statements are equivalent: 
\begin{enumerate}
\item $M(T,Y)$ is an $F$-restriction monoid.
\item ${\mathrm{dom}}(\varphi_t)$ is a principal order ideal for every $t\in  T$.
\item The image of $\varphi$ is contained in the Munn semigroup $T_Y$ of $Y$. 
\end{enumerate}
\end{lemma}

\begin{proof} (1) $\Rightarrow$ (2)   Since $(x,s)\mathrel{\sigma} (y,t)$ holds if and only if $s=t$ and $(x,s)\geq (y,s)$ holds if and only if $x\geq y$, we conclude that for every $t\in T$ there is a maximum element $y\in Y$ such that $(y,t)\in M(T,Y)$. That is, $y$ is the maximum element for which $y\circ t$ is defined or, equivalently, $y$ is the maximum element of ${\mathrm{ran}}(\varphi_t)$. Hence, $y\circ t$ is the maximum element of ${\mathrm{dom}}(\varphi_t)$. The implication (2) $\Rightarrow$ (1) is proved by reversing the arguments. 

The equivalence (2) $\Leftrightarrow$ (3) is immediate.
\end{proof}

\begin{corollary} A $W$-product $W(T,Y)$ is $F$-restriction if and only if $Y$ has an identity if and only if $W(T,Y)$ is a monoid.
\end{corollary}

\subsection{Strong partial actions} 
A left partial action $\cdot$ of a monoid $T$ on a set $Y$ is called {\em strong} if the following requirement holds.
\begin{itemize}
\vspace{0.1cm}
\item[(S)]  For allÊ $s,t\in T$  and  $y\in Y$:  if  $t\cdot  y$  and  $(st) \cdot y$  are defined then 
$s \cdot (t\cdot y)$  is defined.
\vspace{0.1cm}
\end{itemize}
If the above condition is met, we have $s \cdot (t\cdot y)=(st)\cdot y$.
A {\em strong right partial action} is defined dually. 
Strong partial actions of monoids were first considered in~\cite{MS}, where Condition~(S) is a part of the definition of a partial action and the term strong is not used. They were then studied in \cite{GH,H}.
If the monoid $T$ is a group, its left partial action, as defined by (LP1) and (LP2), is a wider notion than the usual partial action of a group \cite{E, KL}, since the latter has to satisfy an additional requirement: for any $g\in G$ and $x\in X$ if $g\cdot x$ is defined, then also $g^{-1}\cdot (g\cdot x)$ is defined and $g^{-1}\cdot (g\cdot x)=x$. A  left partial  action of a group, as defined by (LP1) and (LP2), is strong if and only if it is a usual left partial action. 

A {\em globalization} of a left partial action $\cdot$  of  a monoid $T$ on a set $Y$ consists of (i) a left partial action $\hat{\cdot}$ of $T$ on a set $\overline{Y}$ such that $\hat{\cdot}$ is isomorphic to $\cdot$ and (ii) a left action $*$ of $T$ on a superset $X\supseteq \overline{Y}$, such that for every $y\in \overline{Y}$: $t\,\hat{\cdot}\, y= t*y$ whenever $t\, \hat{\cdot}\, y$ is defined. 
Strong left partial actions are precisely the left partial actions that can be globalized (easy to verify or see \cite{H}). 
A similar definition and remark apply to right partial actions.

A dual concept to strongness, antistrongness, arises for partial actions by injective maps. Namely, if  $\cdot$ is such a left partial action and $\varphi:T\to {\mathcal I}(Y)$, $t\mapsto \varphi_t$, is the corresponding dual prehomomorphism  then the assignment $t\mapsto \varphi_t^{-1}$ defines a dual antiprehomomorphism and thus a right partial action, $\circ$. It is natural to call $\cdot$ {\em antistrong} if $\circ$ is strong.
If $T$ is a group then its (usual) left partial action is strong if and only if it is antistrong, but this is not the case for monoid actions in general. 

\subsection{Strong partial actions and a  condition from \cite{CG}}
Let $S$ be a proper restriction semigroup. We argue that strongness of the underlying partial actions $\cdot$ and $\circ$ of $S/\sigma$ on $E=P(S)$ is equivalent to the following condition ${\mathrm{(EP)^r}}$ and its dual condition ${\mathrm{(EP)^l}}$ arising in  \cite{CG}. 
\begin{enumerate}
\vspace{0.1cm}
\item[${\mathrm{(EP)^r}}$] For all $s,t,u\in S$: if $s\mathrel{\sigma} tu$ then there exists $v\in S$ with $t^+s=tv$ and $u\mathrel{\sigma} v$.
\vspace{0.1cm}
\end{enumerate}

\begin{proposition}\label{prop:char} Let $S$ be a proper restriction semigroup.  Then 
\begin{enumerate}
\item $S$ satisfies condition ${\mathrm{(EP)^r}}$ if and only if $\circ$ is strong. 
\item $S$ satisfies condition ${\mathrm{(EP)^l}}$ if and only if $\cdot$ is strong. 
\end{enumerate} 
\end{proposition}

\begin{proof}  (1) Assume first that ${\mathrm{(EP)^r}}$ holds and show that $\circ$ is strong.  Let $p,q\in S/\sigma$ and $x\in E$ be such that $x\circ (pq)$ and $x\circ p$ are defined. Let, further, $y=x\circ p$ and $s\in pq$, $t\in p$ be such that $x=s^+=t^+$ and $y=t^*$. Take any element $u\in q$. By ${\mathrm{(EP)^r}}$ there is $v\in q$ such that $t^+s=tv$. Since $t^+=s^+$, this yields $s=tv$. Now, letting $e=v^+$, 
this implies $t^+=s^+=(tv^+)^+=(te)^+$. Therefore, $t=t^+t=(te)^+t=te$ and consequently $v^+=e\geq t^*=y$.
Thus $y\circ q$ is defined, so that $\circ$ is strong.

Conversely, suppose that the partial action $\circ$ is strong and show that ${\mathrm{(EP)^r}}$ holds. 
Assume that $s,t,u\in S$ are such that $s\mathrel{\sigma} tu$. Then $s^+t^+\circ \sigma(tu)$ is defined. Since $s^+t^+\circ \sigma(t)$ is obviously defined, too, and since $\circ$ is strong,  it follows that $(s^+t^+\circ \sigma(t)) \circ \sigma(u)$ is defined. Thus there exists $v'\in \sigma(u)$ satisfying $(v')^+\geq s^+t^+\circ \sigma(t) = (s^+t)^*$. Then  
$$(s^+tv')^+=(s^+t(v')^+)^+=  (s^+t)^+=s^+t^+=(t^+s)^+.
$$
Since also $s^+tv' \mathrel{\sigma} tu \mathrel{\sigma} t^+s$, we obtain 
$t^+s=s^+tv'=t(s^+t)^*v'$. Setting $v=(s^+t)^*v'$, we have that $v\in \sigma(u)$ and
$t^+s=tv$, so that ${\mathrm{(EP)^r}}$ holds.

(2) follows by a dual argument.
\end{proof}

 We call a restriction semigroup $S$ {\em left extra proper} ({\em right extra proper}) if the underlying  left partial action $\cdot$  (respectively, the underlying right partial action $\circ$) is strong.  We call $S$ {\em extra proper} if both $\cdot$ and $\circ$ are strong.  Proposition \ref{prop:char} shows that this terminology agrees with that proposed in \cite{CG}.
 
 \subsection{Strong partial actions and {\em FA}-monoids from \cite{FGG}} It is easy to verify that {\em FA}-monoids, considered in \cite[Section 9]{FGG}, are precisely the ample extra proper $F$-restriction monoids.
 
\subsection{Partial actions restricting global actions} 
Let $*$ be an order-preserving left action of a monoid $T$ on a poset $X$ and $Y$ be a subset of $X$, which is an order ideal of $X$ and a meet semilattice under the induced order.  Furthermore, we assume that the induced partial action of $T$ on $Y$ satisfies axioms (A), (B) and (C), so that the semigroup $M(T,Y)$ can be formed. We additionally assume that for every $t\in T$ and $x\in X$
\begin{equation}\label{eq:axiom1} 
\text{ if } x\leq t*y, \text{ where } y\in Y, \text{ then } x=t*z \text{ for some } z\leq y.
\end{equation}

\begin{lemma}\label{lem:products}  For any $x,y\in Y$ and $s\in S$ such that $x\circ s$ is defined the meet $x\wedge s*y$ exists in $X$ and 
$$
x\wedge s*y=s\cdot ((x\circ s)\wedge y),
$$ 
where $\circ$ is the right partial action reverse to $\cdot$.
\end{lemma}

\begin{proof} Since $Y$ is a semilattice and both $x\circ s$ and $y$ belong to $Y$,  the meet $(x\circ s) \wedge y$ exists in $Y$. Let $p= (x\circ s) \wedge y$. Since $p\leq x\circ s$ and $s\cdot (x\circ s)$ is defined, $s\cdot p$ is defined and $s\cdot p\leq x$. From $p\leq y$ we have $s\cdot p=s*p\leq s*y$ and thus $s\cdot p$ is a lower bound for $x$ and $s*y$. Assume that $q\leq x,s*y$.  Since $q\leq s*y$,  by \eqref{eq:axiom1}, there is some $z\leq y$ such that $q=s*z$.  But $z,q\in Y$, so that $q=s\cdot z$ and  $z=q\circ s$. Now, $q\leq x$ implies $q\circ s\leq x\circ s$ and so
$q\circ s\leq (x\circ s)\wedge y=p$. This yields $q\leq s\cdot p$. We have proved that $x\wedge (s\cdot y)$ exists in $X$ and equals $s\cdot p=s\cdot ((x\circ s)\wedge y)$, as requied.
\end{proof}


\subsection{Ultra proper restriction semigroups}
A left partial action $\cdot$ of a monoid $T$ on a set $Y$ will be called a {\em partially defined action}\footnote{The term {\em partial action} would be  more appropriate here, but in this paper it has a different meaning.} if for all $s,t\in T$ and $x\in Y$ the following condition is met:

\vspace{0.1cm}
\begin{enumerate}
\item[(PDA)] $(st)\cdot x$ is defined if and only if $t\cdot x$ and $s\cdot (t\cdot x)$ are defined.
\end{enumerate}

It is clear that the above condition holds if and only if the dual prehomomorphism $\varphi:T\to \mathcal{PT}(Y)$ corresponding to $\cdot$ is in fact a homomorphism. We call a restriction semigroup {\em ultra proper} if its underlying left partial action $\cdot$ is a partially defined action.

\begin{lemma} A restriction semigroup $S$ is ultra proper if and only if its underlying right partial action $\circ$ is a partially defined action.
\end{lemma}
\begin{proof} 
Assume that $\cdot$ is a partially defined action and show that so is $\circ$. Let $x\in P(S)$ and $s,t\in S/\sigma$ be such that that $x\circ st$ is defined. Putting $y=x\circ st$ we get $x=st\cdot y$. Since $\cdot$ is a partially defined action, $t\cdot y$ and $s\cdot (t\cdot y)$ are defined. Let $z=t\cdot y$ and $u=s\cdot z$. Then $y=z\circ t$ and $z=u\circ s$, whence $y=(u\circ s)\circ t$. Thus $\circ$ is a partially defined action. The `if' part follows by symmetry.
\end{proof}

It follows that a left partially defined action of a monoid on a semilattice, satisfying axioms (A), (B), (C), is necessarily both strong and antistrong, and we have the following inclusions of classes of restriction semigroups:
$$
\text{Proper } \supset \text{ Extra proper } \supset \text{ Ultra proper }
$$

\begin{example} {\em By Proposition \ref{prop:w_prod} $W$-products (and in particular semidirect products) of semilattices by monoids are ultra proper.}
\end{example}

\begin{example} {\em Ultra proper restriction semigroups, which are inverse, are precisely semidirect products of semilattices by groups. This is because a monoid homomorphism $\varphi\colon G\to {\mathcal I}(Y)$ with $G$ being a group has its image in the unit group of ${\mathcal I}(Y)$.}
\end{example}

\begin{example}\label{ex:free}{\em The free restriction monoid and the free restriction semigroup  are ultra proper. We verify this, for example, for  $F{\mathcal{RM}}(A)$. The underlying left partial action of $A^*$ on ${\mathcal Y}'$ is given by:  $t\cdot B$ is defined if and only if $t^{-1}\in B$ in which case $t\cdot B=t*B$. Assume that $(st)\cdot B$ is defined. Then $t^{-1}s^{-1}\in B$. Since $B$ is prefix-closed, we have $t^{-1}\in B$, so that $t\cdot B$ is defined. Further, $t^{-1}s^{-1}\in B$ implies $s^{-1}\in t\cdot B$, so that $s\cdot (t\cdot B)$ is defined. Note that the free inverse semigroup and monoid are not ultra proper by the previous example. }
\end{example}


\subsection{Ultra $F$-restriction monoids} \label{subsub:a} 
We call a restriction monoid $S$ {\em ultra $F$-restriction} if it is ultra proper and $F$-restriction. This means that the left partial action $\cdot$ underlying $S$ is a partially defined action and ${\mathrm{dom}}(\varphi_t)$ (and then also ${\mathrm{ran}}(\varphi_t)$)  is a principal order ideal for every $t\in S/\sigma$, where $\varphi\colon S/\sigma\to {\mathcal I}(P(S))$ is the homomorphism corresponding to $\cdot$. In other words, $\varphi$ is a monoid homomorphism from $S/\sigma$ to the Munn monoid $T_{P(S)}$ of the semilattice $P(S)$. 

\begin{example}\label{ex:free1} {\em The free restriction monoid $F{\mathcal{RM}}(A)$ is ultra $F$-restriction. This follows from Example \ref{ex:free} and the fact that ${\mathrm{dom}}(\varphi_t)$ is a principal order ideal of ${\mathcal Y}'$ (recall that the order on ${\mathcal Y}'$ is the reverse inclusion)  generated by $\{1, t_1, t_1t_2, \dots, t_1\cdots t_n\}$, where $t^{-1}=t_1\cdots t_n$ and $t_i\in A$ for all $i$. Note that  $F{\mathcal{RS}}(A)$  is not $F$-restriction since ${\mathrm{dom}}(\varphi_1)={\mathcal Y}$ is not a principal order ideal.}
\end{example}


\subsection{An ultra $F$-restriction cover of a restriction monoid}\label{sub:ultra_f_cover}  The construction in this section is inspired by \cite{FGG} (and an explicit connection  with \cite{FGG} will be stated later on in Remark \ref{rem:link}). Let $S$ be a restriction monoid and $A$ be its generating set as a $(2,1,1,0)$-algebra. Let $T=A^*$ be the free monoid generated by $A$ (in the usual monoid signature $(2,0)$). We put $E=P(S)$ and for $v\in T$ let $\overline{v}$ be the value of $v$ in $S$. For $v\in T$ and $e\in E$ we set
\begin{equation*}\label{eq:act1}
v\cdot e \text{ is defined  if and only if  } \overline{v}^*\geq e,
\end{equation*}
in which case we put $v\cdot e = (\overline{v}e)^+$.

\begin{lemma}\label{lem:important}\mbox{}
\begin{enumerate}
\item The partial map $\cdot$ is a partially defined action of $T$ on $E$ and the monoid $M(T,E)$ is ultra $F$-restriction (note that $M(T,E)$ is even ample since $T$ is cancellative).
\item The map $M(T,E)\to S$ given by $(e,v)\mapsto e\overline{v}$ is a surjective projection separating $(2,1,1,0)$-morphism.
\end{enumerate}
\end{lemma}

\begin{proof}  (1) Assume that $vw\cdot e$ is defined and show that $w\cdot e$ and $v\cdot (w\cdot e)$ are defined. By assumption, $(\overline{vw})^*\geq e$. Therefore,
$$
\overline{w}^*\geq (\overline{v}^*\overline{w})^*=(\overline{v}\,\overline{w})^*=(\overline{vw})^* \geq e.
$$
Thus $w\cdot e$ is defined. To show that $v\cdot (w\cdot e)$ is defined, we verify that $\overline{v}^*\geq (\overline{w}e)^+$. We note that the latter is equivalent to $\overline{w}e=\overline{v}^*\overline{w}e$, which holds because $$
\overline{v}^*\overline{w}e=\overline{w}(\overline{v}^*\overline{w})^*e=\overline{w}(\overline{v}\,\overline{w})^*e=\overline{w}e.
$$

Assume now that $w\cdot e$ and $v\cdot (w\cdot e)$ are defined and show that $wv\cdot e$ is defined. By assumption,
$\overline{w}^*\geq e$ and $\overline{v}^*\geq (\overline{w}e)^+$. Then we obtain that
$$
(\overline{v}\overline{w})^*=(\overline{v}^*\overline{w})^*\geq ((\overline{w}e)^+\overline{w})^*=(\overline{w}e)^*=e.
$$

If both  $w\cdot e$ and $v\cdot (w\cdot e)$ are defined, we have 
$$
v\cdot (w\cdot e)=(\overline{v}(\overline{w}e)^+)^+=(\overline{v}\,\overline{w}e)^+=vw\cdot e.
$$
The remaining axioms of a partial action, as well as that each $\sigma$-class of $M(T,E)$ has the maximum element, are easy to verify.

(2) We first show that the given assignment preserves the multiplication. Let $\circ$ be the right partially defined action reverse to $\cdot$. Then $e\circ v$ is defined if and only if $\overline{v}^+\geq e$ and whenever defined it equals $(e\overline{v})^*$. Using this, we calculate, for any $(e,v), (f,u)\in M(T,E)$: 
$$
(e,v)(f,u)=(v\cdot ((e\circ v)\wedge f),vu)=((\overline{v}(e\overline{v})^*f)^+,vu)=((e\overline{v}f)^+,vu)\mapsto (e\overline{v}f)^+\overline{vu};
$$
$$
e\overline{v}f\overline{u}=e(\overline{v}f)^+\overline{v}\, \overline{u}=(e(\overline{v}f)^+)^+\overline{v}\,\overline{u}=(e\overline{v}f)^+\overline{v}\overline{u}.
$$
So preservation of the multiplication is verified. It is easy to see that $^*$, $^+$ and identity are preserved, too.  That the assignment is projection separating is immediate.
\end{proof}

\subsection{An ultra proper cover of a restriction semigroup}\label{sub:modification}

Let $S$ be a restriction semigroup and let $A$ be its generating set as a $(2,1,1)$-algebra. Let $T=A^*$ and let  $\epsilon$ denote the empty word. Similarly to Section \ref{sub:ultra_f_cover} (and setting $\epsilon$ to act as the identity map), one readily constructs a partially defined action $\cdot$ of $T$ on $E=P(S)$ such that the semigroup $M(T,E)$ can be formed and is ultra proper. Furthermore, the map $\varphi:  M(T,E)\to S$ given by $(e,\epsilon)\mapsto e$ and $(e,v)\mapsto e\overline{v}$, if $v\neq\epsilon$,  is a surjective projection separating $(2,1,1)$-morphism.

The monoid $S^1=S\,\cup \,\{1\}$, where $1\not\in S$, becomes a restriction monoid  if we set $1^*=1^+=1$. We have that $P(S^1)=E^1$ and refer to $S^1$ as the restriction monoid obtained from the restriction semigroup $S$ by {\em adjoining an identity element}. The set $A$ is a generating set of $S^1$ as a $(2,1,1,0)$-algebra. Applying the construction of Section~\ref{sub:ultra_f_cover} we can construct the monoid $M(T, E^1)$ and the cover $M(T, E^1)\to S^1$. It is easy to see that $$M(T, E^1)=M(T,E)\cup (\epsilon, 1)$$ as a set and, moreover, $M(T, E^1)$ is a restriction monoid obtained from the restriction semigroup $M(T,E)$ by adjoining an identity element.


\subsection{Ultra $F$-restriction monoids are the monoids $Y*_m T$ from \cite{FGG}}\label{sub:mprod} In this section we show that a monoid is ultra $F$-restriction if and only if it is $(2,1,1,0)$-isomorphic to a monoid
$Y*_mT$ considered in \cite{FGG,GouldS}. We first recall its definition.  

Let $T$ be a monoid and $Y$ a semilattice with identity, $\epsilon$. Furthermore, let $*$ be a left action and $\bullet$ a right action  of $T$ on $Y$ such that for all $t\in T$ and $x,y\in Y$:
\begin{equation}\label{eq:z1}
t*(x\wedge y)= t*x\wedge t*y, \,\, (x\wedge y)\bullet t=x\bullet t \wedge y\bullet t;
\end{equation}
\begin{equation}\label{eq:z2}
(t*x)\bullet t=\epsilon \bullet t \wedge x, \,\,  t*(x\bullet t)=x\wedge t*\epsilon.
\end{equation}
The actions $*$ and $\bullet$ are then said to form a {\em double action} of $T$ on $Y$. Let
\begin{equation}\label{eq:z3}
Y*_mT=\{(y,t)\in Y\times T\colon y\leq t*\epsilon\}
\end{equation}
and define the multiplication on it by
$$
(x,s)(y,t)=(x\wedge s*y, st).
$$
Further, 
for every $(y,t)\in X*_mT$
we put 
$$
(y,t)^*=(y\bullet t,1), \,\, (y,t)^+=(y,1).
$$

\begin{proposition}[\cite{FGG}]
$Y*_mT$ is a proper restriction monoid with identity $(\epsilon, 1)$ and
$$P(Y*_mT)=\{(y,t)\in Y*_mT\colon t=1\}.$$
The semilattice $P(Y*_mT)$  is order isomorphic to $Y$ via the map $(y,1)\mapsto y$;  $(y,t)\mathrel{\sigma} (x,s)$ if and only if $t=s$ and $(Y*_mT)/\sigma\simeq T$ via the map $\sigma(y,t)\mapsto t$.
\end{proposition}

It is natural to ask if  the partial action underlying $Y*_mT$ has some specific properties caused by the 
symmetry of the double action defining it, a question which we now consider.
Assume that $*$ and $\bullet$ form a double action of $T$ on $Y$. 
Define the following partial map $T\times Y\to Y$:
\begin{equation}\label{eq:we}
t\cdot y \text{ is defined if and only if } y\leq \epsilon\bullet t \text{ in which case set } t\cdot y=t*y.
\end{equation}
For $y\in Y$ let $y^{\downarrow}=\{x\in Y\colon x\leq y\}$ denote the principal order ideal generated by $y$.

\begin{proposition}\label{lem:lem12}\mbox{}
\begin{enumerate}
\item The map $\cdot$ is a partially defined action, which satisfies (A), (B) and (C), so that we can form  the semigroup  $M(T,Y)$. Moreover, $M(T,Y)$ is an ultra $F$-restriction monoid.
\item $M(T,Y)$ and $Y*_mT$ are equal as $(2,1,1,0)$-algebras.
\end{enumerate}
\end{proposition}

\begin{proof} (1) Since $\cdot$ is obtained by restricting a global action $*$, it is a strong partial action. 
To verify (A) and (B), we show that $\varphi_t\colon x\mapsto t\cdot x$ is an order-isomorphism between $(\epsilon\bullet t)^{\downarrow}$ and $(t*\epsilon)^{\downarrow}$. If $t\cdot x$ is defined, we have $t\cdot x\leq t*\epsilon$ as $x\leq \epsilon$ and $*$ is order-preserving by~\eqref{eq:z1}. Assume that $x\leq t*\epsilon$. Then $x=t*(x\bullet t)$ by~\eqref{eq:z2} and thus $x=t\cdot(x\bullet t)$ since $x\bullet t\leq \epsilon\bullet t$. It follows that $x\in {\mathrm{ran}}(\varphi_t)$, so that ${\mathrm{ran}}(\varphi_t)=(t*\epsilon)^{\downarrow}$. Assume that $t\cdot x\leq t\cdot y$. As before, we have $x=(t\cdot x)\bullet t$ and $y=(t\cdot y)\bullet t$. Hence $x\leq y$ as $\bullet$ is order-preserving by~\eqref{eq:z1}. Since $t\cdot (\epsilon\bullet t)$ is defined for $t\in T$, (C) also holds. We are left to verify that $\cdot$ is a partially defined action. Assume that $ts\cdot y$ is defined.  Since
$y\leq \epsilon\bullet ts=(\epsilon \bullet t)\bullet s\leq \epsilon \bullet s$, the element $s\cdot y$ is defined. Then $t\cdot(s\cdot y)$ is defined, too, because $\cdot$ is strong.

(2)  Let $\circ$ be the right partial action converse to $\cdot$. Observe that $x\circ t$ is defined if and only if $x\leq t*\epsilon$, which implies that $M(T,Y)=Y*_mT$ as sets. It is immediate that their unary operations and the identity elements coincide. We verify that the multiplications coincide, too. For this, we verify that $
x\wedge s*y=s*(x\bullet s\wedge y)$ whenever $x\leq s*\epsilon$:
\begin{align*}
s*(x\bullet s\wedge y) & =s*(x\bullet s)\wedge s*y & \text{by } \eqref{eq:z1}\\
& = x\wedge s*\epsilon\wedge s*y & \text{by } \eqref{eq:z2}\\
& = x\wedge s*y &  \text{since } x\leq s*\epsilon.
\end{align*}
 \end{proof}

In the opposite direction, let $\cdot$ be a partially defined left action  of a monoid $T$ on a semilattice $Y$ satisfying axioms (A), (B), (C) such that the monoid $M(T,Y)$ is ultra $F$-restriction. We aim to construct a double action of $T$ on $Y$. Let $\circ$ be the right partially defined action converse to $\cdot$.  For $t\in T$ we let $d_t, r_t\in Y$ be the top elements of  ${\mathrm{dom}}(\varphi_t)$ and ${\mathrm{ran}}(\varphi_t)$, respectively. 

For each $t\in T$ and $y\in Y$ we set
\begin{equation}\label{eq:def}
t*y=t\cdot (y\wedge d_t), \,\, y\bullet t=(y\wedge r_t)\circ t.
\end{equation}

\begin{proposition}\mbox{}\label{lem:assign}
\begin{enumerate}
\item The maps $*$ and $\bullet$ form a double action of $T$ on $Y$. Consequently, we can form the monoid $Y*_mT$.
\item $Y*_mT$ and $M(T,Y)$ are equal as $(2,1,1,0)$-algebras.
\end{enumerate}
\end{proposition}

\begin{proof} (1) First, we verify that $*$ is an action. Let $t,s\in T$ and $y\in Y$. To show that $ts*y=t*(s*y)$ we need to verify that
$$
ts\cdot (y\wedge d_{ts})=t\cdot (s\cdot (y\wedge d_s)\wedge d_t).
$$
By (PDA) $s\cdot (y\wedge d_{ts})$ and $t\cdot (s\cdot (y\wedge d_{ts}))$ are defined and the left hand side of the above equality equals $t\cdot (s\cdot (y\wedge d_{ts}))$. So the needed equality is equivalent to the equality
$$
s\cdot (y\wedge d_{ts})=s\cdot (y\wedge d_s)\wedge d_t.
$$
Denote the left hand side of the above equality by $A$ and the right hand side by $B$. Since $t\cdot A$ is defined, we have $A\leq d_t$. Further, since $ts\cdot d_{ts}$ is defined and $ts\cdot d_{ts}=t\cdot (s\cdot d_{ts})$, we obtain $d_{ts}\leq d_s$. Then $A\leq s\cdot  (y\wedge d_s)$, so that we have proved the inequality $A\leq B$.

To prove that $B\leq A$, we let $x=B\circ s$. Since $B\leq d_t$, it follows that $t\cdot (s\cdot x)$ is defined, so that 
$ts\cdot x$ is defined which implies $x\leq d_{ts}$. Since $s\cdot x=B\leq s\cdot (y\wedge d_s)$, it follows that
$x\leq y\wedge d_s$. Therefore,  $x\leq d_{ts}\wedge y\wedge d_s=y\wedge d_{ts}$, whence
$B=s\cdot x\leq s\cdot (y\wedge d_{ts})=A$, as required. That $\bullet$ is an action is established similarly. 

Let $t\in T$ and $x,y\in Y$.  The first equality in \eqref{eq:z1} holds:
$$t*(x\wedge y)=t\cdot(x\wedge y\wedge d_t)=t\cdot ((x\wedge d_t)\wedge (y\wedge d_t))=t\cdot (x\wedge d_t)\wedge t\cdot (y\wedge d_t)=t*x\wedge t*y.
$$ 
The second equality is verified similarly.

For the first equality in \eqref{eq:z2} we calculate:
\begin{align*}
(t*x)\bullet t & = (r_t\wedge t\cdot (d_t\wedge x))\circ t &\\
& = (t\cdot (d_t\wedge x))\circ t & \text{since } t\cdot (d_t\wedge x)\leq r_t\\
& = d_t\wedge x;\\
(\epsilon\bullet t)\wedge x& =(\epsilon\wedge r_t)\circ t\wedge x=d_t\wedge x.
\end{align*}

The second equality is verified similarly.

(2) It is immediate that $Y*_mT$ and $M(T,Y)$ are equal as sets and that their unary operations and identities coincide. We only verify that the multiplication in $Y*_mT$ coincides with that in $M(T,Y)$. This reduces to verifying that
$x\wedge s\cdot (y\wedge d_s)=s\cdot (x\circ s\wedge y)$ whenever $x\circ s$ is defined: 
\begin{align*}
x\wedge s\cdot (y\wedge d_s)& = s\cdot (x\circ s)\wedge s\cdot (y\wedge d_s) & \\
& = s\cdot (x\circ s\wedge y\wedge d_s) & \\
& = s\cdot (x\circ s\wedge y) & \text{since }x\circ s\leq d_s.
\end{align*}
\end{proof}

\begin{remark}{\em Let $M(T,Y)$ be an ultra $F$-restriction monoid defined by a left partially defined action $\cdot$ of $T$ on $Y$ and let $\circ$ be the right partially defined action converse to $\cdot$. Let, further, $*$ and $\bullet$ be the actions defined in \eqref{eq:def}.  For every $t\in T$ define the maps
\begin{align*} & \varphi_t: d_t^{\downarrow}\to Y,  \,\, x\mapsto t\cdot x;  \,\,  \psi_t: r_t^{\downarrow}\to Y,  \,\,x\mapsto x\circ t;\\
& \tilde{\varphi}_t: Y\to r_t^{\downarrow},  \,\, x\mapsto t*x;  \,\, \tilde{\psi}_t: Y \to d_t^{\downarrow}, \,\,  x\mapsto x\bullet t.
\end{align*}
For every $x\leq d_t$ and $y\in Y$ we then have:
$$
\varphi_t(x)\leq y \text{ if and only if } x\leq  \tilde{\psi}_t(y).
$$
Indeed, $t\cdot x\leq y$ is equivalent to $t\cdot x\leq y\wedge r_t$, which is in turn equivalent to 
$x\leq (y\wedge r_t)\circ t$.
Similarly for every $x\leq r_t$ and $y\in Y$:
$$
\psi_t(x)\leq y \text{ if and only if } x \leq  \tilde{\varphi}_t(y).
$$
This means that the map $\varphi_t$ is a left adjoint to the map $\tilde{\psi}_t$ and the map $\psi_t$ is a left adjoint to the map $\tilde{\varphi}_t$.}
\end{remark}

\begin{remark}\label{rem:link} {\em Using the results of this section, we can present the monoid $M(T,E)$ from Section \ref{sub:ultra_f_cover} in the form $E*_m T$. It is then easy to verify that it coincides (after a change in notation) with the covering monoid considered in the proof of Theorem 7.1 of~\cite{FGG}.} 
\end{remark}

\section{Globalization of a  strong partial action}\label{s:glob} 
Let $T$ be a monoid and $\cdot$ a strong  left partial action of $T$ on a semilattice $Y$ satisfying axioms (A), (B) and (C). In this section we construct a globalization  $*$ of  $\cdot$.
The construction is based on a combination of ideas to be found in \cite{MS, M}. Let $\circ$ be the right partial action converse to $\cdot$.

For $(x,s), (y,t) \in Y\times T$ we set $(x,s) \to (y,t)$ if there is $p\in T$ such that $s=tp$ and   $p\cdot x=y$. 
So we have:
$$
(x,tp) \to (p\cdot x,t); \,\,\, (x,t) \leftarrow (x\circ p, tp)
$$
whenever $p\cdot x$ or $x\circ p$ is defined.

Let $\sim$ be the minimum equivalence relation on $Y\times T$, which contains the relation $\to$. In other words,  $\sim$ is the transitive closure of $\to\cup \leftarrow$. For $A,B\in (Y\times T)/\sim$ we set $A\geq B$ if there are $(x,s)\in A$ and $(y,s)\in B$ such that $x\geq y$.

\begin{lemma}\label{lem:preorder}\mbox{}
\begin{enumerate}
\item \label{i:c1} If $A\geq B$ and $(z,t)\in A$ then there is $(u,t)\in B$ where $z\geq u$.
\item \label{i:c2} The relation $\geq$ is a preorder on $(Y\times T)/\sim$.
\end{enumerate}
\end{lemma}

\begin{proof} (1) Let $(x,s)\in A$ and $(y,s)\in B$ be such that $x\geq y$.  Since $(x,s)\sim (z,t)$, there is a sequence  $(x,s)=(x_0,s_0), (x_1,s_1), \dots, (x_n,s_n)=(z,t)$ in $Y\times T$ such that either $(x_i,s_i)\to (x_{i+1},s_{i+1})$ or  $(x_{i+1},s_{i+1})\to (x_i,s_i)$ for all admissible $i$. Assume that $(x,s)\to (x_{1},s_{1})$. Then there is a factorization $s=s_1q$ such that $x_1=q\cdot x$. It follows that $q\cdot y$ is defined and $(y,s)\to (q\cdot y,s_1)$. We put $y_1=q\cdot y$ and note that $x_1\geq y_1$. Let now $(x_1,s_1)\to (x,s)$. Then there is a factorization $s_1=sp$ such that $x_1=x\circ p$.  Similarly as before, we get that $y\circ p$ is defined and $(y\circ p, sp)\to (y,s)$. We put $y_1=y\circ p$ and note that $y_1\leq x_1$. The statement now follows by induction.

(2) Reflexivity of $\geq$ is obvious and transitivity follows from (1).
\end{proof}

\begin{lemma} \label{lem:ms} Let $(x,s)\sim (y,t)$. Then either both $s\cdot x$ and $t\cdot y$ are defined, in which case $s\cdot x=t\cdot y$, or they are both undefined.
\end{lemma}

\begin{proof} It is enough to consider only the case where $(x,s)\to (y,t)$, since the other case then holds by symmetry and the statement follows by induction. 
Rewriting $(x,s)\to (y,t)$ as $(x,tp)\to (p\cdot x, t)$,  we see that the claim holds since $\cdot$ is strong. 
\end{proof}

\begin{lemma}\label{lem:class} Let $A,B\in (Y\times T)/\sim$ be such that $A\neq B$, $A\leq B$ and $B\leq A$. If $(x,s)\in A$ then $s\cdot x$ is undefined. Consequently, no element of the form $(x,1)$ belongs to $A$.
\end{lemma}

\begin{proof} Let $(x,s)\in A$. By  Lemma \ref{lem:preorder} and since $A\cap B=\varnothing$, there are $(z,s)\in B$ and $(y,s)\in A$ such that $x\gneq z\gneq y$. Assume that $s\cdot x$ is defined. Then by Lemma \ref{lem:ms} $s\cdot y$ is defined and $s\cdot x=s\cdot y$. It follows that $x= (s\cdot x)\circ s = (s\cdot y)\circ s =y$,
which is a contradiction.
\end{proof}

Let $\approx$ be the equivalence on $Y\times T$ given by $(x,s)\approx (y,t)$ if and only if 
$[x,s]_{\sim} \leq  [y,t]_{\sim}$ and $[y,t]_{\sim} \leq  [x,s]_{\sim}$, where $[x,s]_{\sim}$ denotes the $\sim$-class of $(x,s)$. The set $X=(Y\times T)/\approx$ is  partially ordered with the order induced by the preorder on $(Y\times T)/\sim$. By $[x,s]$ we will denote the $\approx$-class of $(x,s)$.
We let $\overline{Y}=\{[y,1]\colon y\in Y\}$.

\begin{lemma}\mbox{}\label{lem:l:a}
\begin{enumerate}
\item \label{i:d1}The map $\theta: y\mapsto [y,1]$ is an order-isomorphism between $Y$ and $\overline{Y}$.
\item \label{i:d2} $\overline{Y}$ is an order ideal of $X$.
\item \label{i:d3} $\overline{Y}$ is a meet semilattice under the induced order on $\overline{Y}$ and $Y$ is isomorphic to $\overline{Y}$ as a meet semilattice via $\theta$. \end{enumerate} 
\end{lemma}

\begin{proof} \eqref{i:d1} Clearly, $\theta$ is surjective. Let $x,y\in Y$.  If $x\leq y$, we have $[x,1]_{\sim}\leq [y,1]_{\sim}$ by the definition of the order on $(Y\times T)/\sim$, so that $[x,1] \leq [y,1]$. Assume that $[x,1] \leq [y,1]$. Then  $[x,1]_{\sim}\leq [y,1]_{\sim}$ and Lemma \ref{lem:preorder} implies that $[x,1]_{\sim}=[z,1]_{\sim}$ for some $z\leq y$. By Lemma \ref{lem:ms} we obtain $x=1\cdot x = 1\cdot z=z$ yielding $x\leq y$.

\eqref{i:d2} Assume that $[x,s]\leq [y,1]$. Then $[x,s]_{\sim}\leq [y,1]_{\sim}$. By Lemma \ref{lem:preorder} we then obtain that $(x,s)\sim (z,1)$ for some $z\leq y$, which implies $(x,s)\approx (z,1)$ and consequently $[x,s]\in \overline{Y}$.

\eqref{i:d3} follows from (1) and (2).
\end{proof}

\begin{lemma}\label{lem:action} Let $[x,s]=[y,p]$. Then $[x,ts]=[y,tp]$ for any $t\in T$.
\end{lemma}

\begin{proof} Since $[x,s]_{\sim}\leq [y,p]_{\sim}$,  we have $[x,s]_{\sim} = [z,p]_{\sim}$ for some $z\leq y$ by Lemma \ref{lem:preorder}. We show that $[x,ts]_{\sim} = [z,tp]_{\sim}$. Assume that
$(x,s)\to (z,p)$. This can be rewritten as $(x,pq)\to (q\cdot x, p)$. But then $(x,tpq) \to (q\cdot x, tp)$, which means that $(x,ts)\to (z,tp)$.  The claim that $[x,ts]_{\sim} = [z,tp]_{\sim}$ now easily follows by induction. Therefore, $[x,ts]_{\sim}\leq [y,tp]_{\sim}$. The opposite inequality can be proved similarly, so that $[x,ts]=[y,tp]$, as required. \end{proof}

Let $t\in T$ and $[y,s]\in X$. We set $t*[y,s] = [y,ts]$. By the preceding lemma this is well defined and thus is an order-preserving left action of $T$ on $X$. 

\begin{lemma}\mbox{}\label{lem:l:b}
\begin{enumerate}
\item \label{i:f1}The induced left partial action of $T$ on $\overline{Y}$ is isomorphic to the left partial action $\cdot$ of $T$ on $Y$.
\item \label{i:f2} If $[x,s] \leq t*[y,1]$ then $[x,s]=t*[z,1]$ for some $z\leq y$.
\end{enumerate}
\end{lemma}

\begin{proof}  \eqref{i:f1} Let $t\in T$ and $y\in Y$. By Lemmas \ref{lem:ms} and \ref{lem:class} we have that $[y,t]=[z,1]$ for some $z$ if and only if $t\cdot y$ is defined.  Assume that $t\cdot y$ is defined. Then  $t*[y,1]=[y,t]=[t\cdot y, 1]$. Assume that $t\cdot y$ is undefined. Then $t*[y,1]\not\in \overline{Y}$ and so the induced partial action is undefined on $[y,1]$.

\eqref{i:f2} follows from Lemma \ref{lem:preorder} and the definition of $*$. 
\end{proof}

\begin{proposition} \label{prop:mult} Let $S$ be a left extra proper restriction semigroup. Let $X$ be the poset and $*$ be the left action of $S/\sigma$ on $X$ obtained by applying the globalization construction to the strong left partial action $\cdot$ underlying $S$. Then for every $x,y\in \overline{P(S)}$ and $s\in S/\sigma$  the meet $x\wedge s*y$ exists in $X$ and the multiplication in the semigroup $M(S/\sigma, \overline{P(S)})$ can be expressed by the formula:
$$
(x,s)(y,t)=(x\wedge s*y,st).
$$ 
\end{proposition}

\begin{proof} By Lemma \ref{lem:l:b}\eqref{i:f2} condition \eqref{eq:axiom1} is satisfied, so the statement follows by Lemma~\ref{lem:products}.
\end{proof}

Proposition \ref{prop:mult} provides a globalized version of  Theorem \ref{th:CG} and may be considered an analogue of the McAlister $P$-theorem \cite{McA, Law} for left extra proper restriction semigroups. A similar statement can be formulated and proved for right extra proper restriction semigroups.

\section{Embedding of  an ultra $F$-restriction monoid $M(A^*,Y)$ into a $W$-product}\label{s:emb}

It is natural to look for conditions on the semigroup $M(T,Y)$, under which the poset $X$ constructed in  the previous section would be a semilattice and the action $*$ would have nice properties. 
In this section we show that this can be achieved if $T$ is a free monoid and $M(T,Y)$ is ultra $F$-restriction. To be precise, our result is formulated as follows:

\begin{theorem}\label{th:main} Let $T=A^*$ be the free $A$-generated monoid and assume that a left partially defined action $\cdot$ of $T$ on a semilattice $Y$ is given such that the semigroup $M(T,Y)$ can be formed and is an ultra $F$-restriction monoid. Let $X$ and $*$ be the poset and the left action of $T$ on $X$ constructed in Section \ref{s:glob}. Then
\begin{enumerate}
\item  \label{i:m1} $X$ is semilattice.
\item \label{i:m2}The $W$-product $W(T,X)$ may be formed.
\item \label{i:m3} $M(T,Y)$ embeds into $W(T,X)$.
\end{enumerate}
\end{theorem}

The remainder of this section will be devoted to the proof of Theorem \ref{th:main}. For $v\in A^*$ by $|v|$ we denote the length of $v$. The empty word is denoted by $1$ and we put $|1|=0$.
The following two lemmas hold under a milder assumption that $M(T,Y)$ is an ultra proper restriction semigroup.

\begin{lemma} \label{lem:x}\mbox{} 
\begin{enumerate}
\item \label{i:x1} Every $\sim$-class $B$ of $Y\times T$ has a unique representative, which we call {\em canonical}, of the form $(x,w)$ where 
$$
|w| = \min\{|u|\colon (y,u)\in B\}.
$$
If $(x,w)$ is a canonical representative and $(y,u)\sim (x,w)$ then $(y,u)=(x\circ t,wt)$ for some $t\in A^*$.
\item \label{i:x2} The equivalences $\sim$ and $\approx$ on $Y\times T$ coincide.
\end{enumerate}
\end{lemma}

\begin{proof} (1) Let $(y,u)\in B$ and $v$ be the longest suffix of $u$ such that $v\cdot y$ is defined and put $x=v\cdot y$. Due to (PDA), $v'\cdot y$ is defined for any suffix $v'$ of $v$. Let $w\in A^*$ be such that $u=wv$ (any of the words $w$ or $v$ may be empty).
We have  $(y,u)\sim (x,w)$.
Assume  that $(z,r)\sim (x,w)$ and show that there is $t\in A^*$ such that $z=x\circ t$ and $r=wt$. By the definition of $\sim$, there is a sequence 
$$
(x,w)=(x_0,w_0), (x_1,w_1),\dots, (x_n,w_n)=(z,r)
$$
such that for each admissible $i$ we have $(x_i,w_i)\to (x_{i+1},w_{i+1})$ or $(x_{i+1},w_{i+1})\to (x_{i},w_{i})$. Note that we necessarily have $(x_1,w_1)\to (x_0,w_0)$ and so $(x_1,w_1)=(x_0\circ t, w_0t)=(x\circ t, wt)$ for some $t\in A^*$. We proceed by induction on $n$. Assume that $(x_i,w_i)=(x\circ t, wt)$. If $(x_{i+1},w_{i+1})\to (x_{i},w_{i})$, it is immediate that 
$(x_{i+1},w_{i+1})$ equals  $(x\circ ts, wts)$ for some $s\in A^*$. Consider now the case where $(x_i,w_i)\to (x_{i+1},w_{i+1})$. By assumption we have $(x_i,w_i)=(x\circ t, wt)$
and then $(x_{i+1},w_{i+1})=(p\cdot (x\circ t), q)$, where $p$ is a suffix of $wt$ and $q$ is determined by $qp=wt$.
Note that $p$ must be a suffix of $t$ for if we assume that $p=w't$ with $w'$ being non empty,  
we obtain that $w'\cdot x$ is defined (since $w't\cdot (x\circ t)$ is defined,  $t\cdot (x\circ t)$ is defined and $\cdot$ is strong), which contradicts the choice of $(x,w)$. This proves  that $(x,w)$ is  canonical, that it is unique and so the claim about the form of any element in its $\sim$-class holds.

(2) We prove that the preorder $\leq$ on $(Y\times T)/\sim$, defined before Lemma~\ref{lem:preorder}, is in fact an order. Indeed, assume that $[y,v]_{\sim}\leq [x,w]_{\sim}$, $[x,w]_{\sim}\leq [y,v]_{\sim}$ and  that $(x,w)$ is a canonical element. Since $[y,v]_{\sim}\leq [x,w]_{\sim}$, we have that $(y,v)\sim (z,w)$ for some $z\leq x$ by Lemma~\ref{lem:preorder}. Similarly, $[x,w]_{\sim}\leq [z,w]_{\sim}$ implies that $(x,w)\sim (x',w)$ for some $x'\leq z\leq x$. But the latter is possible only if $x'=x$ since, as we have proved, any element in $(x,w)_{\sim}$ has the form $(x\circ t, wt)$ for some $t\in A^*$. 
\end{proof}

\begin{example} {\em Let $F{\mathcal{RM}}(A)=M(A^*, {\mathcal Y}')$ be the free restriction monoid over $A$. The relation $\sim$ on ${\mathcal Y}'\times A^*$ is given by $(B,v)\sim (B',v')$ if and only if $v*B=v'*B'$, where we remind that $v*B=\{\mathrm{red}(vb)\colon b\in B\}$. It follows that $[B,v]_{\sim}\mapsto v*B$ defines a bijection between the sets $({\mathcal Y}'\times A^*)/\sim$ and  ${\mathcal Q}'$. It is easy to see that the order on $({\mathcal Y}'\times Z^*)/\sim$  corresponds to the anti-inclusion order on ${\mathcal Q}'$ under this bijection
and the canonical elements of ${\mathcal Y}'\times A^*$ are precisely the elements $(B,v)$ such that either $v=1$, or, otherwise, $z^{-1}\not\in B$ with $z$ being  the last letter of $v$. For an arbitrary element $(B,v)$ the canonical element in its $\sim$-class is the element $(w\cdot B, u)$, where $w$ is the longest suffix of $v$ such that $w^{-1}\in B$ and $v=uw$. Recall that the elements of ${\mathcal X}'$ can be interpreted as finite connected subgraphs in the Cayley graph of $F{\mathcal{G}}(A)$  and the elements of ${\mathcal Y}'$ as such subgraphs containing the origin (see \cite{S1,S} for details). Then the elements of ${\mathcal Q}'$ correspond to the finite connected subgraphs of the form $t*C$ where $C\in {\mathcal Y}'$ and $t\in A^*$.  If ${\Gamma}$ is such a subgraph then the canonical element $(B,v)$ corresponding to it under the bijection between ${\mathcal Q}'$ and $({\mathcal Y}'\times A^*)/\sim$  is determined as follows: $v$ the vertex of $\Gamma$ which is the closest to the origin (such a vertex exists since the graph is a tree) and $B=v^{-1}*{\Gamma}$.
}
\end{example}

Recall that the action $*$ of $T$ on $X$ which globalizes $\cdot$ is given by $t*[x,w]=[x,tw]$. 
Let the map $\alpha_t:X\to X$ be given by $x\mapsto t*x$. 

\begin{lemma}\mbox{}\label{lem:action_prop}
\begin{enumerate}
\item For every $t\in T$ ${\mathrm{ran}}(\alpha_t)$ is an order ideal of $X$.
\item $\alpha_t: X\to {\mathrm{ran}}(\alpha_t)$ is an order-isomorphism.
\end{enumerate}
\end{lemma}

\begin{proof}
(1) Assume that $[y,v]\leq t*[x,w]=[x,tw]$. Then $[y,v]=[z,tw]$, where $z\leq x$, so that $[y,v]=t*[z,w]$. 

(2) It is immediate that $*$ is order-preserving. We now
assume that $t*[x,w]\leq t*[y,v]$ and show that $[x,w]\leq [y,v]$.
Let $q$ be the longest suffix of $tw$ such that $q\cdot x$ is defined and let $tw=pq$. Then
$$
[x,tw]=[x,pq]=[q\cdot x, p]
$$ 
and the element $(q\cdot x,p)$ is canonical.
Similarly, $[y,tv]=[q'\cdot y, p']$ with the element $(q'\cdot y, p')$ being canonical, where $q'$ is the longest suffix of $tv$ such that $q'\cdot y$ is defined. Assume that $tv=p'q'$. Since $[q\cdot x, p]\leq [q'\cdot y, p']$, there is some $z\leq y$ such that
$$
[q\cdot x, p]=[q'\cdot z, p'].
$$
Canonicity of $(q\cdot x,p)$ implies that
$$
(q'\cdot z, p')=((q\cdot x)\circ m, pm)
$$
for some $m\in T$. It follows that 
$$
q\cdot x=mq'\cdot z \text{ and } pm=p'.
$$
Then we have $tw=pq$ and $tv=pmq'$. We consider two possible cases:

{\em Case 1.} Assume that $|p|\geq |t|$. Then $|q|\leq |w|$ and $|mq'|\leq |v|$, which implies that
$w=rq$ and $v=rmq'$ for some suffix $r$ of $p$. Then we have:
$$
[x,w]=[x,rq]=[q\cdot x,r]=[mq'\cdot z, r]=[q'\cdot z, rm]\leq [q'\cdot y, rm]=[y,rmq']=[y,v].
$$

{\em Case 2.} Assume that $|p|< |t|$. Then $|q|>|w|$ and $|mq'|>|v|$. Thus $q=kw$ and $mq'=kv$ for some suffix $k$ of $t$. Since $q\cdot x$ defined,  $w\cdot x$ is also defined and we have
$$
w\cdot x  = (q\cdot x)\circ k = (mq'\cdot z)\circ k\leq (mq'\cdot y)\circ k = (kv\cdot y)\circ k=v\cdot y.
$$
Therefore, $[x,w]=[w\cdot x,1]\leq [v\cdot y,1]=[y,v]$, which completes the proof.
\end{proof}

From now on we assume that $M(T,Y)$ is an ultra $F$-restriction monoid. For $v\in T$ let $d_v$ and $r_v$ denote the top elements of ${\mathrm{dom}}(\varphi_v)$ and  ${\mathrm{ran}}(\varphi_v)$, respectively. 

\begin{lemma}\label{lem:meet_x}
$X$ is a semilattice. The meet on $X$ is calculated by the rule:
\begin{equation}\label{eq:meet_x}
[e,v]\wedge [f,u] = [v'\cdot (e\wedge d_{v'})\wedge u'\cdot (f\wedge d_{u'}),k],
\end{equation}
where $k$ is the longest common prefix of $v$ and $u$ and $v=kv'$, $u=ku'$.
\end{lemma}

\begin{proof}
Let $[e,v], [f,u]\in X$.  Let, further, $k$ be the longest common prefix of $v$ and $u$ and assume that $v=kv'$ and $u=ku'$. We put
$$
g=v'\cdot (e\wedge d_{v'})\wedge u'\cdot (f\wedge d_{u'})
$$
and show that $[g,k]=[e,v]\wedge [f,u]$. To verify that $[e,v]\geq [g,k]$, it is enough to show that
$[g,k]= [e',v]$ for some $e'\leq e$. Since $g\leq v'\cdot (e\wedge d_{v'})$, the element $g\circ v'$ is defined and $g\circ v'\leq e\wedge d_{v'}$. Therefore,
$$
[g,k]=[g\circ v', kv']\leq [e\wedge d_{v'}, v]\leq [e,v].
$$
Similarly, $[g,k]\leq [f,u]$.

Suppose that $[h,t]\leq [e,v],[f,u]$, where we assume that the element $(h,t)$ is canonical. Then there are $e'\leq e$ and $f'\leq f$ such that $[h,t]=[e',v]=[f',u]$. Since $(h,t)$ is canonical, Lemma \ref{lem:x}\eqref{i:x1} implies that $$
e'=h\circ p, v=tp, f'=h\circ q, u=tq
$$
for some $p,q\in T$. Hence $t$ is a common prefix of $v$ and $u$, so that $k=tl$ for some $l\in T$ by maximality of $k$. We also have $p=lv'$ and $q=lu'$. Since $h\circ p$ is defined,  $h\circ l$ and $(h\circ l)\circ v'$ are defined and
$e'=h\circ p=(h\circ l)\circ v'$ by (PDA). Similarly, $h\circ q=(h\circ l)\circ u'$. Then
$$
[e',v]=[h\circ lv', tlv']=[h\circ l, tl] = [h\circ l, k].
$$
Since $h\circ p\leq e$ and $h\circ p = (h\circ l)\circ v' \leq d_{v'}$, we obtain $h\circ p\leq e\wedge d_{v'}$. Therefore,
$$
h\circ l = v'\cdot (h\circ p)\leq v'\cdot (e\wedge d_{v'}).
$$
Similarly,  $h\circ l\leq u'\cdot (f\wedge d_{u'})$.
Therefore, $h\circ l\leq g$, so that  $[h,t]=[h\circ l, k]\leq [g,k]$, and the proof is complete.
\end{proof}

By Lemmas \ref{lem:action_prop} and \ref{lem:meet_x} we may form the $W$-product $W(T,X)$.  As it is shown in \ Subsection \ref{sub:w}, we have  $W(T,X)=M(T,X)$. 
To complete the proof of Theorem~\ref{th:main} we show that $M(T,\overline{Y})$ embeds into $W(T,X)$. Let $\hat{\cdot}$ be the left partially defined action of $T$ on $\overline{Y}$ isomorphic to $\cdot$ via $[y,1]\mapsto y$. 
Then the action $*$ is an extension of $\hat{\cdot}$ and the right partially defined action $\bullet$ reverse to $*$ is an extension of the right partially defined action $\circ$ reverse to $\hat{\cdot}$. 
Therefore,
\begin{multline}\label{eq:aux20}
M(T,\overline{Y})=\{([y,1], t)\in \overline{Y}\times T\colon   [y,1]\circ t \text{ is defined} \}= \\
\{([y,1], t)\in \overline{Y}\times T \colon  [y,1] \bullet t \text{ is defined and }[y,1] \bullet t \in \overline{Y}\}\subseteq M(T,X)=W(T,X).
\end{multline}
By Proposition \ref{prop:mult} we obtain that $M(T,\overline{Y})$ is a $(2,1,1)$-subalgebra of $M(T,X)$.

\begin{example}\label{ex:free4}{\em  Consider the free restriction monoid $F{\mathcal RM}(A)= M(A^*, {\mathcal Y}')$. As already remarked, the semilattice
${(\mathcal Y}'\times A^*)/\sim$ is order isomorphic to ${\mathcal Q}'$. It is easy to see that the action of $A^*$ on ${(\mathcal Y}'\times A^*)/\sim$, which globalizes $\cdot$, is isomorphic to the action $*$ of $A^*$ on ${\mathcal Q}'$ given by $v*E=\{{\mathrm{red}}(ve)\colon e\in E\}$. It follows that the monoid $W(T,X)$ from Theorem \ref{th:main}\eqref{i:m2} is isomorphic to the monoid $W(A^*,{\mathcal Q}')$ and the embedding of $F{\mathcal RM}(A)$ into $W(A^*,{\mathcal Q}')$, produced by the proof of Theorem \ref{th:main}, coincides with the (left hand version of the) embedding constructed by Szendrei in \cite{S1}. }
\end{example}

\begin{remark} {\em We remark that Lemma \ref{lem:meet_x} and thus also Theorem \ref{th:main} can be formulated and proved in a slightly more general setting where $M(T,X)$ is an ultra proper restriction semigroup and ${\mathrm{dom}}(\varphi_t)$ is a principal order ideal for all $t\neq 1$. Examples where this setting arises are the free restriction semigroup $F{\mathcal{RS}}(A)$ and also the semigroup $M(T,E)$ from Section \ref{sub:modification}. The resulting construction of embedding of  $M(T,E)$ into a $W$-product then generalizes Szendrei's embedding of $F{\mathcal{RS}}(A)$ into a $W$-product \cite{S1}.}
\end{remark}

We deduce the following result which is equivalent to the main result of \cite{S1}, but the formulation below emphasizes the fact that the covering semigroups and monoids can be chosen ultra proper or ultra $F$-restriction.

\begin{theorem}\label{th:ms1}\mbox{}
\begin{enumerate}
\item   Every restriction monoid $S=\langle A \rangle$ has an ultra $F$-restriction (ample) cover $M(A^*, P(S))$ which $(2,1,1)$-embeds into a $W$-product $W(A^*, X)$.
\item Every restriction semigroup $S=\langle A\rangle$ has an ultra proper (ample) cover $M(A^*, P(S))$ which $(2,1,1)$-embeds into a $W$-product $W(A^*, X')$.
\end{enumerate}
\end{theorem}

\begin{proof} (1) follows from Lemma \ref{lem:important} and Theorem \ref{th:main}. 

(2) The semigroup $M(A^*, P(S))$, constructed in Section \ref{sub:modification}, is an ultra proper cover of $S$ and is a $(2,1,1)$-subalgebra of  the ultra $F$-restriction monoid $M(A^*, P(S)^1)$. The latter can be $(2,1,1)$-embedded into a $W$-product $W(A^*, X')$ by Theorem \ref{th:main}.
\end{proof}

\section{An embedding of a restriction semigroup into a quotient of $W(A^*,X)$}\label{s:quot}

Let $S=\langle A\rangle$ be a restriction monoid  with $E=P(S)$.  Let, further, $M(A^*,E)$ be the ultra $F$-restriction cover of $S$ from Section \ref{sub:ultra_f_cover} and $W(A^*,X)$ be the $W$-product, produced by the proof of Theorem \ref{th:main}.  In this final section we construct a projection separating $(2,1,1)$-congruence $\kappa$ on $W(A^*,X)$ such that $S$ $(2,1,1)$-embeds into  $W(A^*,X)/\kappa$. This yields a new and simpler proof of the main result of \cite{S} that any restriction semigroup embeds into an almost left factorizable restriction semigroup.

We set $W=W(A^*,X)$ and define 
$W_1$ to be the subset of $W$ consisting of the elements, which can be written in the form $([e,p],pq)$. We may assume that the element $(e,p)$ is canonical as otherwise we have that $p=rt$ and $e=f\circ t$, where $(f,r)$ is the canonical element equivalent to $(e,p)$, so that $([e,p],pq)=([f,r],rtq)$.

We define an auxiliary relation $\gamma$ on $W_1$ by setting $x\mathrel{\gamma} y$ if and only if
$$
x=([e,p],pq), y=([e,p],pr) \text{ and } e\overline{q}=e\overline{r},
$$
where the element $(e,p)$ is canonical. 
For $x=([e,p],pq)\in W_1$ with $(e,p)$ being canonical we put 
$${\mathrm{inv}}(x)=e\overline{q}.$$

For $x,y\in W$ we put
$x\mathrel{\kappa} y$ if and only if  $x=y$ or  $x,y\in W_1$ and $x\mathrel{\gamma} y$.

\begin{lemma}\label{lem:kappa}
The relation $\kappa$ is a $(2,1,1)$-congruence on $W$.
\end{lemma}

\begin{proof} It is immediate that this is an equivalence relation. Let $x\mathrel{\kappa}y$.  We may assume that $x,y\in W_1$ and that $x=([e,p],pq)$, $y=([e,p],pr)$, where $(e,p)$ is a canonical element and $e\overline{q}=e\overline{r}$.
It is immediate that $x^+=([e,p],1)=y^+$. We show that $x^*=y^*$.
We have 
$$
x^*=([e\circ q,1],1), \,\, y^*=([e\circ r,1],1).
$$
The elements $(e\circ q,1)$ and $(e\circ r,1)$ are clearly canonical. The needed equality follows from
$$
e\circ q=(e\overline{q})^*=(e\overline{r})^*=e\circ r.
$$ 

We now show that $\kappa$ is stable with respect to the multiplication on the left and on the right.
Let $z=([f,a],b)\in W$.
We first  show that $xz\mathrel{\kappa} yz$. Applying \eqref{eq:meet_x} we write
\begin{equation}\label{eq:aux11}
xz=([e,p]\wedge [f,pqa], pqb)=([e(\overline{qa}f)^+,p],pqb]).
\end{equation} 
Note that
\begin{align*}
e(\overline{qa}f)^+ & = (e(\overline{qa}f)^+)^+ & \text{since } x^+=x \text{ for } x\in E\\
& = (e\overline{q}\,\overline{a}f)^+ & \text{by }\eqref{eq:consequences}\\
\end{align*}
and similarly $e(\overline{ra}f)^+ =  (e\overline{r}\,\overline{a}f)^+$. Since $e\overline{q}=e\overline{r}$, it follows that $e(\overline{qa}f)^+= e(\overline{ra}f)^+$, let us denote this element by $e'$. Then
$$
xz=([e',p],pqb)\in W_1 \text{ and } yz=([e',p],prb)\in W_1.
$$
Let $p=p_1p_2$ be a factorization such that the element $(p_2\cdot e',p_1)$ is canonical.  We have 
\begin{align*}
{\mathrm{inv}}(xz) & =(p_2\cdot e')\overline{p_2qb} & \\
& = (\overline{p_2}e')^+\overline{p_2}\overline{qb} & \\
& = \overline{p_2}e'\overline{q}\,\overline{b} & \text{by Lemma }\ref{lem:lem1}\eqref{i:b2}
\end{align*}
and similarly ${\mathrm{inv}}(yz)=\overline{p_2}e'\overline{r}\,\overline{b}$.
Our assumption that $x\mathrel{\kappa} y$ and  the inequality $e'\leq e$ imply the equality $e'\overline{q}=e'\overline{r}$. We obtain that
$xz\mathrel{\kappa}yz$.

We now show that $zx\mathrel{\kappa} zy$. We have
$$
zx=([f,a],b)([e,p],pq)=([f,a]\wedge [e,bp],bpq).
$$
Let $c$ be the longest common prefix of $a$ and $bp$, so that $a=ca'$ and $bp=cb'$. Then 
$$
[f,a]\wedge [e,bp]=[(\overline{a'}f)^+(\overline{b'}e)^+,c],
$$
and thus
$$
zx=([(\overline{a'}f)^+(\overline{b'}e)^+,c], cb'q)\in W_1.
$$
Substituting $q$ with $r$ in the above calculation for $zx$, we obtain
$$
zy=([(\overline{a'}f)^+(\overline{b'}e)^+,c], cb'r)\in W_1.$$

Let $c=kl$, where the element $(l\cdot((\overline{a'}f)^+(\overline{b'}e)^+),k)$ is canonical. Then
\begin{align*}
{\mathrm{inv}}(zx)& =(\overline{l}(\overline{a'}f)^+(\overline{b'}e)^+)^+\overline{lb'q} & \\
& = \overline{l}(\overline{a'}f)^+(\overline{b'}e)^+\overline{b'q} & \text{by Lemma } \ref{lem:lem1}\eqref{i:b2}  \text{ since }\overline{l}(\overline{a'}f)^+(\overline{b'}e)^+\leq \overline{l}\\
& = \overline{l}(\overline{a'}f)^+\overline{b'}e\overline{q} & \text{by Lemma } \ref{lem:lem1}\eqref{i:b2}  \text{ since }\overline{b'}e\leq \overline{b'}.
\end{align*}
Hence ${\mathrm{inv}}(zx)={\mathrm{inv}}(zy)$, so that $zx\mathrel{\kappa} zy$, as required.
We have verified that $\kappa$ is a $(2,1,1)$-congruence on $W$.
That $\kappa$ is projection separating is immediate from its definition.
\end{proof}

\begin{theorem}
[Szendrei \cite{S}] Every restriction semigroup is $(2,1,1)$-embeddable into a $(2,1,1)$-morphic image of a $W$-product of a semilattice by a monoid.
\end{theorem}

\begin{proof} Let $S=\langle A\rangle$ be a restriction semigroup and  $M(A^*,E)$  the ultra proper cover of $S$ from Section \ref{sub:modification}. As it is explained in Section \ref{sub:modification}, by adjoining an identity to $M(A^*,E)$ we obtain an ultra $F$-restriction monoid $M(A^*,E^1)$, and $M(A^*,E)$ is a $(2,1,1)$-subalgebra of $M(A^*,E^1)$. Let $W=W(A^*,X)$ be the $W$-product produced by the proof of Theorem~\ref{th:main} out of the monoid $M(A^*,E^1)$ and $\kappa$ be the congruence on $W$ constructed in this section. The definition of $\kappa$ implies that $([x,1],s)\mathrel{\kappa} ([y,1],t)$ if and only if $x\overline{s}=y\overline{t}$. Since the quotient of $M(A^*,\overline{E})$ over this congruence is isomorphic to $S$, the latter embeds into~$W/\kappa$. 
\end{proof}

\section*{Acknowlegement}

The author thanks Victoria Gould and  the anonymous referee  for many useful comments and suggestions.


\vspace{0.2cm}

\noindent
Institute Jo\v zef \v Stefan,\\ 
Jamova cesta 34, SI-1000, Ljubljana, SLOVENIA\\
e-mail: {\tt ganna.kudryavtseva\symbol{64}ijs.si}\\
and\\
Institute of Mathematics, Physics and Mechanics, \\
Jadranska ulica  19, SI-1000, Ljubljana, SLOVENIA\\
\end{document}